\documentclass[12pt]{article}

\usepackage{amsmath} 
\usepackage{amssymb} 
\usepackage{amsthm}
\usepackage{hyperref}
\usepackage{tikz}
\usepackage{fullpage}
\usepackage[all]{xy}
\usetikzlibrary{matrix,arrows}

\newtheorem{theorem}{Theorem}

\newtheorem{proposition}[theorem]{Proposition}

\newtheorem{theodef}[theorem]{Theorem and Definition}

\theoremstyle{definition}
\newtheorem{definition}{Definition}
\newtheorem{example}{Example}

\theoremstyle{remark}
\newtheorem*{remark}{Remark}

\title{G\"odel's Completeness Theorem and Deligne's Theorem}
\author{Benjamin Frot}
\date{}

\begin{document}
\maketitle

These notes were written for a presentation given at the university Paris VII in January 2012. The goal was to explain a proof of a famous theorem by P. Deligne about coherent topoi (coherent topoi have enough points) and to show how this theorem is equivalent to G\"odel's completeness theorem for first order logic. Because it was not possible to cover everything in only three hours, the focus was on Barr's and Deligne's theorems. This explains why the corresponding sections have been given more attention. Section 1 and the Appendix were added in an attempt to make this document self-contained and understandable by a reader with a good knowledge of topos theory.

A coherent topos is a topos which is equivalent to a Grothendieck topos that admits a site $(C,J)$ such that $C$ has finite limits and there exists a base of $J$ with finite covering families. In order to prove
that any coherent topos has enough points, it is enough to show that for any coherent topos $\mathcal{E}$ there is a surjective geometric morphism from a topos with enough points to $\mathcal{E}$. As sheaf topoi over topological spaces have enough points, any surjective geometric morphism from some $Sh(X)$, with $X$ a topological space, to $\mathcal{E}$ is sufficient. These are provided by Barr's theorem. This is the approach that will be taken here.

The rest of the document is organised as follows : Section 1 contains a sequence of definitions and theorems necessary for a good understanding of the statement of G\"odel's completeness theorem. It also recalls a well known result about Stone spaces that will be useful in Section 3. Section 2 and 3 are dedicated to the proofs of Barr's and Deligne's theorems, respectively. Finally, section 4 gives an overview of the correspondence between G\"odel's and Deligne's theorems.

The proofs themselves come mainly from \cite{MM, topost} but the interested reader will also find relevant information in G.C. Wraith's tutorial \cite{wraith} as well as in \cite{costes} (in French) and \cite{Reyes}. 

\section{First Definitions and Theorems}

\subsection{Completeness Theorem}

This section contains a sequence of definitions necessary for a good understanding of G\"odel's completeness theorem. This should hopefully make it easier for the reader to see the analogies
between the "classical" definitions given here and their, less usual, categorical counterparts.

\begin{definition}{First order language}\newline

	A (multisorted) \textit{first order language} $L$ is a language made of:
	\begin{enumerate}
		\item a collection of \textit{sorts} $X,Y,\dots$;
		\item a collection of \textit{relation symbols} $R,S,\dots$;
		\item a collection of \textit{function symbols} $f,g\dots$;
		\item a collection of \textit{constants} $c,d,\dots$.
	\end{enumerate}
\end{definition}

\begin{definition}{Terms} \newline

	Let $L$ be a first-order language as described in the previous definition. 
	Then \textit{terms} are built using the following rules:
	\begin{enumerate}
		\item If $x$ is a variable of sort $X$, then it is a term $t$ of sort $X$;
		\item If $c$ is a constant of sort $X$, then it is a term $t$ of sort $Y$;
		\item If $t_1, t_2, \dots, t_n$ are terms of sorts $X_1,X_2,\dots,X_n$ respectively, and if
			$f: X_1 \times X_2 \times \dots \times X_n \to Y$ is a function symbol, then $f(t_1,t_2,\dots,t_n)$ is
			a term of sort $Y$.
	\end{enumerate}

\end{definition}

\begin{definition}{Atomic Formula} \newline

	Let $L$ be a first-order language. \textit{Atomic formulas} of $L$ are defined as follows:
	\begin{enumerate}
		\item If $R \subseteq X_1 \times X_2 \times \dots \times X_n$ and if $t_1,t_2,\dots,t_n$ are of sorts $X_1, X_2, \dots, X_n$ respectively,
			then $R(t_1,t_2, \dots, t_n)$ is an atomic formula;
		\item If $t$ and $t'$ are both of sort $X$ then $t = t'$ is an atomic formula;
		\item $\top$ and $\bot$ are both atomic formulas.
	\end{enumerate}

\end{definition}

	Atomic formulas can be combined to give more complex formulas. Here, the case of interest is the one of geometric formulas.

\begin{definition}{Geometric Formula} \newline

	A \textit{geometric formula} is a formula built from atomic formulas, $\exists, \vee$ and $\wedge$.

\end{definition}

\begin{definition}{Sequent (Geometric)} \newline

	Let $L$ be a geometric language and $\Gamma, \Delta$ be finite sets of formulas of $L$ with free variables in the finite set $V$.
	A \textit{sequent} is an expression of the form:
	\[ \Gamma \vdash \Delta^V . \]

\end{definition}

\begin{definition}{Geometric Theory} \newline

	A \textit{geometric theory} $T$ is a set of sequents called the \textit{axioms} of $T$. 

\end{definition}

In such theories, theorems are derived using rules of the same kind as the ones of Gentzen's sequent calculus. These theories being
geometric, only the rules containing $\wedge$, $\vee$ and $\exists$ are retained. 
For an explicit description of these rules see \cite{costes}.

Geometric theories (sometimes called coherent theories) are especially well-behaved since they are preserved and reflected by geometric morphisms.
However, usual mathematics make an extensive use of $\neg, \forall$ and $\Rightarrow$ and are expressed in a \textit{classical language} in the
setting of \textit{classical theories}.

\begin{definition}{Classical Language} \newline

	A \textit{Classical language} is a geometric language $L$ to which the connectives $\neg, \forall, \Rightarrow$ are added, as well as
	the corresponding rules of Gentzen's sequent calculus. Note $L^+$ the classical language associated to the geometric language $L$. 

\end{definition}

\begin{definition}{Classical Theory} \newline

	A \textit{classical theory} $T^+$ associated to a geometric theory $T$ is a theory with the same axioms as $T$, only they are seen in $L^+$.

\end{definition}

Given that geometric theories are interesting to work with because of their behavior with respect to geometric morphisms, it is important to know
when a theorem of a classical theory can be proven in its geometric counterpart. Barr's theorem indirectly answers this question, this explains
why we choose to call it ``Barr logic'' here. It will be proved in the next section using topos theoretic tools.

The following comes from \cite{costes}:
\begin{theorem}{``Barr logic``}\newline \label{barr_logic}

	Let $L$ be a geometric language and let $S$ be a sentence of $L$.

	\begin{center}
		\framebox{If $S$ is a theorem of $T^+$ then it is a theorem of $T$. }
	\end{center}

\end{theorem}

G.C. Wraith \cite{wraith} expresses the same theorem as the following meta-theorem:
\begin{center}
	\textit{
	If a geometric sentence is deducible from a geometric theory in classical logic, with the axiom of choice, then it is also deducible from it
	intuitionistically.}
\end{center}

Barr's theorem constitutes a first stepping-stone towards the proof of the completeness theorem. 
In what follows, let $T$ be a classical theory expressed in a language $L$.

\begin{remark}
	Are assumed known the notions of \textit{model} and \textit{interpretation}. A very short introduction is given in the appendix (\ref{modeles}).
\end{remark}

\begin{definition}{Satisfied sequent} \newline \label{satisfied}

	Let $\alpha = \Gamma \vdash \Delta^V$ be a sequent. $\alpha$ is satisfied in an interpretation $M$ of $T$ if
	\begin{center}
		\[M(\bigwedge_{\phi \in \Gamma} \phi, V) \leq M(\bigvee_{\psi \in \Delta} \psi, V) \].
	\end{center}

\end{definition}

\begin{theorem}{G\"odel / Deligne} \newline

	Let $\alpha$ be defined as in Definition \ref{satisfied} and let $T$ be a geometric theory.
	\begin{center}
		\framebox{
			$\alpha$ is provable in $T$ if and only if $\alpha$ is satisfied in all models of $T$.
		}
	\end{center}

\end{theorem}

\subsection{Stone Spaces}

This section is dedicated to the statement of an important result linking the category of sheaves over some
complete boolean algebra to the category of sheaves over the corresponding Stone space. Indeed,
for a complete boolean algebra $B$, there exists a surjective geometric morphism $f : X \to Sh(B)$ which preserves finite
epimorphic families and which is such that $X$ is a \textit{topological space}. This $X$ happens to be $Stone(B)$. As will be shown later, topological spaces have the desirable property of "having enough points", a result that will be useful in the proof of Deligne's theorem.

The theory of Stones spaces is beyond the scope of this document but a detailed enough overview of the topic can be found in \cite{MM} IX.10.

\begin{proposition}\label{stone}

	Let $B$ be a complete boolean algebra. There exists a surjective geometric morphism 
	\[ \phi: \mathcal{O}(Stone(B)) \to B \]
	which induces an embedding 
	\[ i : Sh(B) \to Sh(Stone(B)). \]

	Furthermore, $i_\ast : Sh(B) \to Sh(Stone(B))$ preserves finite epimorphic families.

	(Note: $\mathcal{O}(-)$ is the poset $(\mathcal{O}(X), \subseteq)$ of opens of $X$ seen as a category. More, generally this defines the \textit{frame} associated to some
	locale).
\end{proposition}

Recall that a finite epimorphic family is a set of arrows $\{f_i : X_i \to Y\}_{i=0}^{n}$ such that if $g \circ f_i = h \circ f_i, \forall i$ then $g = h$.
More specifically, in the context of the previous theorem, finite epimorphic families are of the form 
\[\alpha_k : \{F_k \to F\}_{k=0}^{n} \]
where $F$ and the $F_k$ are sheaves of $Sh(B)$.

\section{Barr's Theorem}

\begin{theorem}{Barr's Theorem}\label{barr}\newline

	If $\mathcal{E}$ is a Grothendieck topos. Then, there exists a complete boolean algebra $B$ and a surjective geometric morphism
	$Sh(B) \to \mathcal{E}$.
\end{theorem}

\begin{proof}(\cite{MM})\newline

		$\mathcal{E}$ is a Grothendieck topos and therefore, by definition, it is equivalent to a category $Sh(C, J)$ over some site $(C,J)$.
		Let us fix such a site $(C,J)$ and show the following two points in order:
		\begin{enumerate}
			\item There exists a locale $X$ and a surjective geometric morphism $Sh(X) \twoheadrightarrow Sh(C,J)$;
			\item There exists a surjection of locales $Y \to X$ such that $\mathcal{O}(Y) = B$ is a complete boolean algebra.
		\end{enumerate}

	 \begin{enumerate}
		 \item In order to prove the existence of $X$ and of the surjective geometric morphism $Sh(X) \twoheadrightarrow Sh(C,K)$ we construct them.

			 Define $String(C)$ as follows:
			\begin{itemize}
				\item The objects of $String(C)$ are sequences of arrows of the form 
					$s = (C_n \xrightarrow{\alpha_n} C_{n-1} \to \ldots \xrightarrow{\alpha_1} C_0)$;

				\item Take $s,t \in \mathcal{O}b(String(C))$ with $s$ defined as above. There is an arrow $t \to s$ if and only if
					$t$ is of the form $C_{n + p} \to \dots \to C_n \xrightarrow{\alpha_n} \dots \xrightarrow{\alpha_1} C_0$.
			\end{itemize}

			$\mathcal{O}b(String(C))$ can equivalently be seen as a poset $(\mathcal{O}b(String(C)),\leq)$ in which $t \leq s$ if and only if there is an arrow $t \to s$. Thus, if one had
			a Grothendieck topology $K$ on $String(C)$, $Sh(String(C), K)$ would be a localic topos. 

			To achieve this, let $\pi$ be the functor $String(C) \to C$ which is the assignment of
			\begin{itemize}
				\item $\pi(s) = C_n$ on objects;
				\item $\pi(t \leq s) = \alpha_{n + 1} \circ \dots \circ \alpha_{n + p}$ on arrows.
			\end{itemize}

			Now, if $s$ is an object of $String(C)$ and $U$ is a sieve on $s$, define $U$ to be a covering sieve (in the topology $K$) if and only if
			for any $t$ s.t. $t \leq s$ the set
				\[\{\pi(t' \leq t) | t' \in U\}\] 
				covers $\pi(t)$ according to the topology $J$ defined above.

			Rephrased:
			\begin{center}
				\framebox{$U \in K(s)$ $\Leftrightarrow$ $\forall t \leq s, \{\pi (t' \leq t) | t' \in U\} \in J(t)$.}
			\end{center}

			It is straightforward to check, and we admit, that $K$ is indeed a Grothendieck topology and thus that $Sh(String(C), K)$ is a localic topos.

			It must now be shown that $\pi$ induces a well-defined geometric morphism and that it is surjective:

			\begin{enumerate}
				\item \textit{$\pi$ induces a well-defined geometric morphism $Sh(String(C), K) \to Sh(C, J)$}.\newline

					Start by recalling the following theorem:

					\begin{theorem}

						Let $\pi : D \to C$ be a functor which has the property of \textit{covering lifting} (clp), \textit{i.e.} which is such that:
						\[
						\forall d \in D, \forall R \in J(\pi d), \exists U \in K(d)~ s.t. ~\pi U \subseteq R.
						\]
					then 
					\begin{itemize}
						\item $\pi$ induces a geometric morphism $f: Sh(D, K) \to Sh(C, J)$
						\item $\forall F \in Sh(C, J)$ $f^\ast (F) \cong a (F \circ \pi)$ (where $a$ is the left-adjoint of the inclusion
							$i : Sh(D, K) \to PreSh(D, K)$).
					\end{itemize}
					\end{theorem}

					To check that the hypotheses of the theorem are met, take $s = \{C_n \xrightarrow{\alpha_n} \dots \xrightarrow{\alpha_1} C_0 \}$ and $R$ a $J-$sieve
					covering $\pi s$. 
					Set, $U = \{t' | t' \leq s, \pi(t' \leq s) \in R\}$ 
					(more explicitly: $U = \{(\alpha_{n + m}, \dots, \alpha_n, \dots, \alpha_1) \in String(C) | \alpha_{n + 1} \circ \dots \circ \alpha_{n +m} \in R\}$).
					By definition, $\pi U \subseteq R$. 
					
					It remains to be shown that $U \in K(s)$, or more precisely that for any $t \leq s$, the set 
					$\{ \pi(t' \leq t), t' \in R\}$ covers $\pi(t)$. 
					Fix $t \leq s$ and set $R' = g^\ast R = \{ g | cod(g) = \pi(t), \pi(t \leq s)g \in R \}$. But the stability axiom for $J$ says 
					exactly that $R'$ is a covering sieve for $\pi(t)$ in $C$. 
					Now, notice that $R'$ is contained in $\pi(U \cap \{ t' | t' \leq t \})$ which implies that it covers
					$\pi(t)$. Finally, remember the definition of $K$ to see that $U$ covers $s$ in $String(C)$. 

					Conclusion: $\pi$ induces a geometric morphism $f: Sh(String(C), K) \to Sh(C, J)$.

				\item \textit{$f$ - the geometric morphism induced by $\pi$ - is surjective.}\newline

					A sufficient and necessary condition for a geometric morphism, $f : Sh(D,K) \to Sh(C,J)$ say, to be surjective is:
					\begin{center}
						For any $J-$sheaf $E \in Sh(C, J)$, $f^\ast$ induces an injective geometric morphism $Sub(E) \to Sub(f^\ast E)$.
					\end{center}
				
					Furthermore, we have the following result:
						\begin{theodef}
						Let $P$ be a presheaf over $\mathcal{D}$ and $A \subseteq P$, a subpresheaf of $P$. For any object $D \in \mathcal{D}$, and
						any element $d \in P(D)$, note $d \cdot g := P(g)(d)$. Then,
						\begin{center}
							\framebox{$S_{d, A} = \{ g: D' \to D | d \cdot g \in A(D')\}$ is a sieve over $D$}. 
						\end{center}
						Moreover, the subpresheaf $A$ is closed if and only if $\forall D \in \mathcal{D}, \forall d \in P(D)$,
						$S_{d, A}$ covers $D$ $\Rightarrow$ $d \in A(D)$.
						Finally, if $E \in Sh(C, J)$, then $ClSubPr(E \circ \pi) \cong SubSh(f^\ast E) \cong SubSh(a (E \circ \pi))$.
					\end{theodef}

						Admit just for a second that the triangle
						\begin{center}
							\begin{displaymath}
								\xymatrix{
								&SubSh(E) \ar[ld]_{} \ar[rd]_{f^\ast_E} &\\
								ClSubPr(E \circ \pi) \ar[rr]_{\sim} && SubSh(f^\ast E) }
							\end{displaymath}
						\end{center}
						commutes. Then, since $\pi$ is obviously surjective on objects, $SubSh(E) \to ClSubPr(E \circ \pi), B \mapsto B \circ \pi$ would be injective and, thanks to
					the isomorphism which is at the bottom of the triangle, this part of the proof would be over.

					The commutativity of that triangle requires $SubSh(E) \to ClSubPr(E \circ \pi)$ to be well-defined, \textit{i.e.}
					that for any $B \subseteq E$ the presheaf $B \circ \pi$ is closed in $E \circ \pi$. Commutativity is then given by $a$.

					To prove this, take $s \in String(C)$ and $d \in (E \circ \pi)(s)$ such that $S_{d, B \circ \pi} = \{g: s' \to s | d \cdot g \in B \circ \pi (s') \}$ covers $s$.
					Then, it is obvious that $\pi(S_{d, B \circ \pi})$ covers $\pi(s)$. 
					Moreover, the arrows $\pi g$ are such that $d \cdot g \in B(\pi t)$ and thus are in the sieve $S_{d, B} = \{ h: c' \to \pi(s) | d \cdot h \in B(c')\}$. 
					Consequently, this sieve covers $\pi(s)$ because it contains a sieve covering $\pi(s)$. To conclude, observe that since $B$ is a sub-sheaf of $E$, 
					it is closed in $E$ and $d \in B(\pi s)$.
			\end{enumerate}

		 \item Show that there is a surjection of locales $Y \to X$ such that $\mathcal{O}(Y) = B$ is a complete boolean algebra.
			 \begin{remark}
				 If $Y$ is a locale, we write $\mathcal{O}(Y)$ the corresponding frame.
			 \end{remark}

			 Let $X$ be a locale, then it is possible to define the operations $\Rightarrow$ and $\neg$ in $\mathcal{O}(X)$. For $U$ and $V$ elements of this frame, define:

			 \begin{itemize}
				 \item $U \Rightarrow V = \bigvee\{W \in \mathcal{O}(Y) | W \wedge U \leq V \}$;
				 \item $\neg U = (U \Rightarrow 0)$;
			 \end{itemize}
			 thus turning $\mathcal{O}(Y)$ into a complete Heyting algebra.

			 The set $\mathcal{O}(X_{\neg \neg})$ of fixed points of $\neg \neg : \mathcal{O}(X) \to \mathcal{O}(X)$ is defined by 
			 $\{ U \in \mathcal{O}(X) | \neg \neg U = U \}$. This is a complete boolean algebra and $\neg \neg$ is surjective on $\mathcal{O}(X_{\neg \neg})$.

			\begin{remark}
                The proofs of these assertions are mainly computational and are not of great
                interest here.
			\end{remark}
				
			For any $U \in \mathcal{O}(Y)$, define $(X - U) := \mathcal{O}(X - U) \cong \{ V \in \mathcal{O}(Y) | V \geq U \}$. 
			
				This defines a ``sub-locale'' of $X$. Recast in terms of frames, the embedding $g^{-1} : \mathcal{O}(X) \rightarrowtail \mathcal{O}(X - U)$ is then given by
			 $g^{-1}(W) = W \vee U$.

			 Let \[Y = \coprod_{U \in \mathcal{O}(X)} (X - U)_{\neg \neg}. \] Here, objects are seen as locales and therefore the coproduct (of locales) should be seen as
			 a product of frames. As all the components of this product are complete boolean algebras, $\mathcal{O}(Y)$ is also a complete boolean algebra.

			 Remains the proof of existence of a surjective morphism $Y \to X$, or equivalently, of an injective morphism $\mathcal{O}(Y) \to \mathcal{O}(X)$.
			 For each $U$, set $p_U : (X - U)_{\neg \neg} \to (X - U) \to X$ which defines $p : Y \to X$ on each of the components of the coproduct. Now, assume, that $U$ and 
			 $V \in \mathcal{O}(X)$ are such that $U \neq V$. Without loss of generality, assume further, that $V \geq U$. Then $p_U^{-1}(U) = 0$ while $p_U^{-1}(V) \neq 0$
			 because $V \geq U$. This proves the injectivity of the frame morphism and concludes the proof of Barr's theorem.
	 \end{enumerate}
	\end{proof}

\section{Deligne's theorem}

\subsection{Coherent Topoi}

\begin{definition}{Coherent Topos} \newline \label{toposcoherent}

	A \textit{coherent topos} is a topos which is equivalent to a topos $Sh(C,J)$ such that $C$ has finite limits and $J$ has a base $K$ whose covering families
	are finite.

\end{definition}

\begin{definition}{Point of a topos} \newline

	Let $\mathcal{E}$ be a topos (possibly elementary). A \textit{point} of $\mathcal{E}$ is any geometric morphism $\mathbf{Sets} \to \mathcal{E}$.
\end{definition}

\begin{remark}
	To see that this definition is fairly intuitive, simply look at the case of topological spaces.
	A point of a topological space is a morphism $x : \{\ast\} \to X$ in $Top$. But
	$Sets = Sh(\{ \ast \})$ is the terminal object of the category of Grothendieck topoi.
\end{remark}

\begin{example}{Topological space}\newline

	Let $(X, T)$ and $(Y, T')$ be topological spaces and let $Sh(X), Sh(Y)$ be the sheaf topoi over $\mathcal{O}(X)$ and $\mathcal{O}(Y)$ respectively.
	The set of geometric morphisms $f : Sh(X) \to Sh(Y)$ is in bijection with continuous functions $X \to Y$. If $X = \{\ast \}$ then one sees
	that points of $Y$ are in bijection with points of $Sh(Y)$. Consequently, in the case of topological spaces, a point in the usual sense gives a point 
	in the corresponding sheaf topoi.

	Furthermore, for each sheaf $F$ over $Sh(X)$ and each point $x$ in $X$, define the $stalk$ of $F$ in $x$ by
	 \[F_x = \lim_{\to_{x \in U}} F(U).\]
	Then, for each morphism of sheaves $f: E \to D$ call $f_x$ the morphism $ E_x \to D_x$. Further details can be found in \cite{MM} II.6.
\end{example}

\begin{definition}{To have enough points} \newline \label{assezdepoints}

	Let $\mathcal{E}$ be a topos and $\alpha, \beta : E \to D$ be two distinct parallel arrows. 
	If there exists a point $p$ of $\mathcal{E}$ such that $p^\ast(\alpha) \neq p^\ast(\beta)$ then $\mathcal{E}$ \textit{has enough points}.
\end{definition}

\begin{example}
	Any sheaf topos over a topological space has enough points. 
\end{example}

\begin{proposition}

	Let $\mathcal{E}$ be a topos with enough points, $\mathcal{F}$ a Grothendieck topos and $f : \mathcal{E} \to \mathcal{F}$ a 
	surjective geometric morphism. Then $\mathcal{F}$ has enough points.
\end{proposition}

\begin{proof}
	Let $\alpha, \beta : A \to B$ be two parallel morphisms in $\mathcal{F}$ such that $\alpha \neq \beta$. Then, by surjectivity of
	of $f$, we have $f^\ast(\alpha) \neq f^\ast(\beta)$ and there is a $p: Ens \to \mathcal{E}$ such that $p^\ast \circ f^\ast(\alpha) \neq p^\ast \circ f^\ast(\beta)$.
	Geometric morphisms being stable by composition, $f \circ p : Ens \to \mathcal{F}$ is a point of $\mathcal{F}$ and we have
	$(f \circ p)^\ast = p^\ast \circ f^\ast (\alpha) \neq p^\ast \circ f^\ast (\beta) = (f \circ p)^\ast(\beta)$. 
	(To see this, notice that $p$ and $f$ are both geometric, that $p^\ast \circ f^\ast$ preserves finite limits and that 
	$Hom_{Ens}(p^\ast \circ f^\ast(Y), X) \cong Hom_{\mathcal{E}}(f^\ast(Y), p_\ast(X)) \cong Hom_{\mathcal{F}}(Y, f_\ast \circ p_\ast (X))$.) 

\end{proof}
\subsection{Deligne's Theorem}

\begin{theorem}{Deligne's theorem}\newline

	\begin{center}
		\framebox{
		Coherent topoi have enough points.}
	\end{center}
\end{theorem}
\begin{proof}

	Considering the results obtained earlier, showing that for any coherent topos $\mathcal{E}$ there exists a surjective geometric morphism from $\mathcal{E}$ to $Sh(X)$ where $X$ is
	a topological space is enough to prove that coherent topoi have enough points. 
	This is precisely what will be done here with $X = Stone(B)$, $B$ being a complete boolean algebra.

	Take any complete boolean algebra $B$. Then, by Barr's theorem (\ref{barr}), there is a surjective geometric morphism $f : Sh(B) \twoheadrightarrow \mathcal{E}$.
	Now, suppose that, for any choice of such a $B$, we can find a geometric morphism $g : Sh(Stone(B)) \to \mathcal{E}$ such that
					\begin{center}
							\begin{displaymath}
								\xymatrix{
								Sh(B) \ar@{^{(}->}[d]_{i} \ar[r]_{f} &\mathcal{E}\\
								Sh(Stone(B))  \ar @{.>}[ru]_{g}& }
							\end{displaymath}
						\end{center}

	commutes. By commutativity of the triangle, $g$ would be surjective and Deligne's theorem would be proved.

	The rest of the proof focuses on the existence of $g$, starting by recalling a theorem which establishes a correspondence between
	some continuous left-exact functors and geometric morphisms.

	\begin{theorem}\label{corres}

		Let $\mathcal{F}$ be a topos which has small colimits. Let $\mathcal{E} \cong Sh(C,J)$, with $C$ a category with finite limits. Then

		\begin{itemize}
			\item There is a correspondence between continuous left-exact functors $A : C \to \mathcal{F}$ and geometric morphisms $f: \mathcal{F} \to Sh(C, J)$.
			\item For any $c \in C$ and $G \in Sh(C, J)$, we have $f_\ast(G)(c) \cong Hom(A(c), G)$.
		\end{itemize}
	\end{theorem}

	Let $\mathcal{E} = Sh(C,K)$ be a coherent topos, $K$ being the basis for a Grothendieck topology as in Definition \ref{toposcoherent}, and let
	$f: Sh(B) \to \mathcal{E}$ be a geometric morphism.

	According to Theorem \ref{corres}, there exists a continuous left-exact functor $A : C \to Sh(B)$ which corresponds to $f$. Set $A' = i_\ast \circ A$. 
	Note, that because $\mathcal{E}$ is coherent, $K$ is compact (the topology generated by $K$, to be more precise). Moreover, $i$ preserves finite epimorphic families and $A$ is continuous. All this implies the continuity of $A'$.

	\begin{remark}
		It seems that there is confusion on the definition of a \textit{continuous functor}. Here continuous means ``cover-preserving''.
		(see: \url{http://ncatlab.org/nlab/show/continuous+functor}) 
	\end{remark}

	As a composition of two left-exact functors $A'$ is continuous left-exact and Theorem \ref{corres} yields the existence of corresponding 
	geometric morphism $g : Sh(Ston(B)) \to \mathcal{E}$ whose direct image is given by $g_\ast(F)(c) \cong Hom(A'(c), F)$,
	where $F$ is a sheaf of $Sh(Stone(B))$ and $c$ an object of $C$. 

	In particular, take $G$, a sheaf on $B$. Then we have the following equivalences:
	\begin{center}
		$\begin{array}{lclr}
			g_\ast \circ i_\ast G(c) & \cong & Hom(A'(c), i_\ast(G)) & \\
			             &\cong & Hom(i^\ast A'(c), G) &  $(by definition of adjointness)$\\
									 &\cong& Hom(i^\ast i_\ast (A(c)), G) &$(by definition of $A'$)$ \\
									 &\cong& Hom(A(c), G) & $($i_\ast i^\ast \cong id$, $i$ being an embedding.)$\\
		\end{array}$
	\end{center}

	But, as stated above, there is an equivalence between $f_\ast(G)(c)$ and $Hom(A(C), G)$. Thus, $g$ makes the diagram commute up to natural isomorphism.
\end{proof}

\section{Applications to finitary first order logic}

This section dives into the details of the correspondence between Deligne's and G\"odel's theorem, sketching the proof of: ``Deligne's implies G\"odel's''. 
More precisions on this correspondence as well as the one between Barr's theorem and Theorem \ref{barr_logic} can be found in \cite{topost, Reyes}.

\subsection{Syntactic site associated to a geometric theory}

To each geometric theory it is possible to associate a site in a systematic fashion. It turns out that geometric morphisms with codomain the sheaf topos on that
site are in correspondence with the models of the theory.

For what follows, let $T$ be a geometric theory expressed in a language $L$. A category $C(T)$ is associated to $T$ as follows:

\begin{itemize}
	\item The objects of $C(T)$ are equivalence classes $[\phi]$ of formulas of $L$. If $\phi$ and $\psi$ are two formulas with free variables in
		$X = (x_1, \dots, x_n)$ and $X' = (x_1', \dots, x_n')$ respectively and s.t. $X$ and $X'$ have variables of same sorts
		then write $\phi \sim \psi$ if $\{x | \phi(x)\}^M = \{x' | \psi(x') \}^M$ for any model $M$ of $T$. Intuitively,
		$\phi \sim \psi$ if $\psi \vdash \phi$ and $\phi \vdash \psi$.
	\item Arrows of $C(T)$ also are equivalence classes of formulas. A morphism $[\sigma; X, Y] : [\phi, X] \to [\psi, Y]$
		is represented by formulas $\sigma(x_1, \dots, x_n, y_1, \dots, y_m)$ ($X, Y$ being the free variables of $\sigma$) 
		such that:
		\begin{enumerate}
			\item $\sigma(X, Y) \vdash \phi(X) \wedge \psi(Y)$,
			\item $\phi(X) \vdash \exists y_1, \dots, \exists y_m ~ \sigma(X, y_1, \dots, y_m)$,
			\item and \[\sigma(X, Y) \wedge \sigma(X, y_1', \dots, y_n') \vdash \bigwedge_{i = 1}^m (y_i = y_i')\]
		\end{enumerate}
		are theorems of $T$.
	\item The identity $[\phi, X] \to [\phi', X']$ is the arrow $[\phi'(X') \wedge \phi(X) \wedge x_1 = x_1' \wedge x_2 = x_2' \wedge \dots \wedge x_n = x_n']$ 
	\item Composition $[\phi, X] \xrightarrow{\sigma} [\psi, Y] \xrightarrow{\tau} [\theta, Z]$ is given by 
		$[\exists Y (\sigma(X, Y) \wedge \tau(Y, Z))]$.
\end{itemize}

We admit that $C(T)$ is a category and put on $C(T)$ the Grothendieck topology $J(T)$ defined by:
Given $[\sigma_i X_i, Y] : [\phi_i, X_i] \to [\psi, Y]$, the family $\{[\sigma_i, X_i, Y]\}_{i = 1}^n$ covers $[\psi, Y]$ if and only if for any model
$M$ of $T$ 
\begin{center}
	\[\forall Y, \psi(Y) \vdash \bigvee_{i=1}^n \exists X_i \phi_i( X_i, Y)\] 	
\end{center}
is a theorem of $T$.

This topology is sub-canonical, which, in this context, will mainly mean that presheaves of the form $Hom(-,X)$ are sheaves.

$(C(T), J(T))$ is particularly interesting because of the following theorem of Joyal and Reyes.
\begin{theorem}{Joyal-Reyes}\newline

	Let $\mathcal{E}$ be a Grothendieck topos. Recall that $Mod(\mathcal{E},T)$ is the category of $T-$models in $\mathcal{E}$.
	Then
		\[Hom(\mathcal{E}, Sh(C(T), J(T))) \cong Mod(\mathcal{E}, T)\].
\end{theorem}

Thus, to any geometric morphism in $Sh(C(T), J(T))$ it is possible to associate a model of $T$. One of these morphisms is especially interesting: the identity.
Call $U(T)$ the model corresponding to the identity. Then, any model is, up to isomorphism, the inverse image of $U(T)$ by some geometric morphism. 

\subsection{Deligne/G\"odel}

Using what has just been done, it is finally possible to illustrate the correspondence between Deligne's and G\"odel's theorem.
Indeed, look at the following sequence of equivalences:
\small
$
\begin{array}{lclr}
		(\phi \vdash \psi$ is a theorem$)  & \Leftrightarrow &   (\phi \vdash \psi$ is satisfied in the model $U_T$ of $Sh(C(T), J(T))) & (1)\\
	             & \Leftrightarrow & (\phi \vdash \psi$ is satisfied in all models of $T$ in any topos$) &(2)\\
							 & \Leftrightarrow &  (\phi \vdash \psi$ is satisfied in all models of $T$ in $Sets) &(3)\newline
\end{array}$
\normalsize

$(2)$ is a consequence of Yoneda's embedding and of the canonicity of $J(T)$ \cite{MM} X.7. . While this is not trivial to prove, it will not be done here as the focus is on the role of
of Deligne's theorem in the proof of completeness.

Assuming that $(1)$ and $(2)$ hold, our goal is to show that $(\phi \vdash \psi$ is satisfied in all models of $T$ in $Sets$) $\Rightarrow
(\phi \vdash \psi$ is satisfied in $U_T$).

Let $p : Set \to Sh(C(T), J(T))$ be a point of the classifying topos associated to $T$. Then $M = p^\ast(U_T)$ is a model of $T$ and the sequent $\phi \vdash \psi$ 
is satisfied in $M$. According to Definition \ref{satisfied} this means that:
		\[M(\bigwedge_{\phi \in \Gamma} \phi, V) \leq M(\bigvee_{\psi \in \Delta} \psi, V) \]
or, equivalently,

		\[p^\ast \circ U_T(\bigwedge_{\phi \in \Gamma} \phi, V) \leq p^\ast \circ U_T(\bigvee_{\psi \in \Delta} \psi, V) \]

This is true for any point $p$ of $Sh(C(T), J(T))$ which is a coherent topos. According to Deligne's theorem, this means that:
		\[ U_T(\bigwedge_{\phi \in \Gamma} \phi, V) \leq U_T(\bigvee_{\psi \in \Delta} \psi, V) \],
and concludes the proof.

\subsection*{Acknowledgements}
I would like to thank Alain Prout\'e who first suggested this talk. He was also the one who 
brought \cite{costes} to my attention, article that proved really helpful. His advice and availability were very precious to me both
during the preparation of this talk and afterwards.   

\appendix
\section{Interpretations and Models}\label{modeles}

We will deal here with the general case and let the reader define the corresponding the notions in the special case of $Sets$.
More can be found in \cite{Reyes}.

Set $L$, a geometric language, and $\mathcal{E}$, a topos.

\begin{definition} ($\mathcal{E}-$interpretation of a geometric language)\newline

	An \textit{interpretation} $M$ of $L$ in $\mathcal{E}$ is a function which assigns
	\begin{enumerate}
		\item to each sort $X$ of $L$, an object $M(X) := s^M$ of $\mathcal{E}$;
		\item to each function symbol $f : X_1 \times X_2 \times \dots \times X_n \to Y$, a morphism $M(f) := f^M : X_1^M \times \dots \times X_n^M \to Y^M$;
		\item to each relation symbol $R \subseteq X_1 \times \dots \times X_n$, a subset $M(R) := R^M$ of $X_1^M \times \dots \times X_n^M$;
		\item to each constant $c$ of sort $X$ an element $M(c) := c^M$ of $X^M$.
	\end{enumerate}
\end{definition}

	To a term $t(x_1, \dots, x_n)$ of sort $Y$ whose variables are among the $x_i$ of sort $X_i$ we associate an arrow:
			\begin{center}
				$M(t) := t^M : X_1^M \times \dots \times X_n^M \to Y^M$.
			\end{center}

	To a formula $\phi(x_1, \dots, x_n)$ we associate a subobject of $X_1^M \times \dots \times X_n^M$: $M(\phi) := \{(x_1, \dots x_n) | \phi\}^M$;.

	Atomic formulas $t(x_1, \dots, t_n) = t'(x_1, \dots, x_n)$ are interpreted with the help of the equalizer
	\begin{center}
		$M(t = t') := \{(x_1, \dots, x_n) | t = t' \}^M \rightarrowtail X_1^M \times \dots \times X_n^M \rightrightarrows Y^M$.		
	\end{center}

  $\wedge$ and $\vee$ are interpreted, unsurprisingly, by $\{(x_1, \dots, x_n) | \phi \wedge \psi\}^M = \{(x_1, \dots, x_n\ | \phi\}^M \wedge \{(x_1, \dots, x_n\ | \psi\}^M$ 
	and  $\{(x_1, \dots, x_n) | \phi \vee \psi\}^M = \{(x_1, \dots, x_n\ | \phi\}^M \vee \{(x_1, \dots, x_n\ | \psi\}^M$ respectively. 

	Finally, here is the interpretation of $\exists$. Set $\pi: X_1^M \times \dots \times X_n^M \times X^M \to X_1^M \times \dots \times X_n^M$. As a morphism of $\mathcal{E}$,
	$\pi$ has a left-adjoint $\exists_\pi : Sub(X_1^M \times \dots \times X_n^M \times X^M) \to Sub(X_1^M \times \dots \times X_n^M)$. Then, define
	\begin{center}
		$\{(x_1, \dots, x_n) | \exists x \in X \phi(x_1, \dots, x_n, x)\}^M = \exists_\pi \{(x_1, \dots, x_n, x) | \phi(x_1, \dots, x_n, x) \}^M$.
	\end{center}

\begin{definition}(Model)\newline

	Let $L$ be a geometric language and let $T$ be a theory expressed in $L$. Let $M$ be an interpretation of $T$ in $\mathcal{E}$.
	$M$ is a \textit{model} of $T$ in $\mathcal{E}$ if for any axiom $\Gamma \vdash \Delta^V$ of $T$, we have:
	\begin{center}
		\[M(\bigwedge_{\phi \in \Gamma} \phi, V) \leq M(\bigvee_{\psi \in \Delta} \psi, V) \].
	\end{center}
\end{definition}

Thus, it is possible to define the subcategory (of the category of interpretations of $L$) $Mod(\mathcal{E}, T)$ of models of $T$ in $\mathcal{E}$ in the following way:
\begin{enumerate}
	\item Objects of $Mod(\mathcal{E}, T)$ are the models of $T$ in $\mathcal(E)$;
	\item Arrows of $Mod(\mathcal{E}, T)$ are homomorphisms $H : M \to M'$ such that for each sort of $X$, arrows $H_X: X^M \to X^{M'}$ make the diagrams:
		\begin{center}
			\begin{displaymath}
				\xymatrix{
				R^M \ar@{^{(}->}[rr]_{} \ar[d] && X_1^M \times \dots \times X_n^M \ar[d]_{H_{X_1 \times \dots \times X_n}}\\
				R^{M'} \ar@{^{(}->}[rr]_{} && X_1^{M'} \times \dots \times X_n^{M'} }
			\end{displaymath}
			\begin{displaymath}
				\xymatrix{
				X_1^M \times \dots \times X_n^M \ar[d]_{H_{X_1 \times \dots \times X_n}} \ar[rr]^{f^M} && Y^M \ar[d]_{H_Y}\\
				X_1^{M'} \times \dots \times X_n^{M'} \ar[rr]_{f^{M'}} &&  Y^{M'}}
			\end{displaymath}
			\begin{displaymath}
				\xymatrix{
				1  \ar@{=}[d]  \ar[rr]^{c^M} && X^M \ar[d]_{H_X}\\
				1 \ar[rr]_{c^{M'}} &&  X^{M'}}
			\end{displaymath}

			commute,

\end{center}
		where $R$ a relation symbol, $f$ a function symbol and $c$ a constant.
\end{enumerate}

We conclude this appendix by a very important result which explains why so much attention is given to geometric theories.

\begin{proposition}
	Let $T$ be a geometric theory. Let $\mathcal{E}, \mathcal{F}$ be topoi and $f : \mathcal{F} \to \mathcal{E}$ be a geometric morphism.
	Then $f$ induces a functor $F^\ast : Mod(\mathcal{E}, T) \to Mod(\mathcal{F}, T)$.
\end{proposition}

 \bibliographystyle{alpha}
\bibliography{bibliography}

\end{document}